\documentclass[12pt]{amsart}
\usepackage{amsmath,amssymb,amsfonts,amsthm,mathrsfs}
\usepackage{color}
\usepackage{xfrac}
\usepackage{tikz}
\usepackage{palatino}

\newtheorem*{theorem*}{Theorem}
\newtheorem{theorem}{Theorem}
\newtheorem{proposition}[theorem]{Proposition}
\newtheorem{lemma}[theorem]{Lemma}

\theoremstyle{definition}
\newtheorem{example}{Example}

\newtheorem{remark}{Remark}

\renewcommand{\epsilon}{\varepsilon}
\renewcommand{\phi}{\varphi}

\def\<{\left\langle}
\def\>{\right\rangle}

\title[Reciprocal Transformations]{Reciprocal Transformations and Their Discrete Maharam Extensions}
\author{Chris Johnson}
\address{Department of Mathematics and Computer Science, Western
  Carolina University, Cullowhee, North Carolina 28723}
\email{ccjohnson@gmail.com}
\subjclass[2020]{Primary 37A40; Secondary 37E05, 37E35}
\keywords{Ergodic theory, nonsingular transformations, interval
  exchange transformations, translation surfaces}
\date{\today}

\begin{document}
\maketitle

\begin{abstract}
  We introduce two abstract constructions for building new measurable
  dynamical systems from existing ones and study their ergodic
  properties.  The first of these constructions, a ``reciprocal
  transformation,'' produces a type of non-singular transformation
  where the measures of subsets are distorted in a simple way.  We
  then introduce the ``discrete Maharam extension'' which associates
  an infinite measure-preserving transformation to each reciprocal
  transformation.  We give some preliminary results about the ergodic
  theory of each of these constructions, mention ongoing work, as well
  as conjectures and questions for future research.
\end{abstract}

\section{Introduction}
In this paper we consider two simple constructions and explore ergodic
theoretic properties of these constructions.  We will also mention
some possible applications of these constructions to the theory of
interval exchanges and translation surfaces.

The first construction we introduce considers non-singular
transformations that arise by composing a measure-preserving
transformation with a special map we call a ``scaling involution'' to
produce what we will refer to as a ``reciprocal transformation.''
This composition gives us a map which no longer preserves measure, but
where the measure is distorted in an easy-to-understand way.  This
gives us a family of non-singular transformations which we believe are
similar enough to measure-preserving transformations that in some
particular instances the dynamics can be easily understood.

The second construction gives a method for associating a
measure-preserving transformation, albeit on a space of infinite
measure, to each reciprocal transformation.  As we will see, ergodic
properties of these infinite transformations are reflected in the
dynamics of the finite non-singular transformations.

While the definitions and theorems given in this article are expressed
in terms of abstract ergodic theory, we will consider as running
examples the special cases when the initial measure-preserving
transformation is an interval exchange transformation, and the scaling
involution is a piecewise affine involution of an interval.  Indeed,
the initial goal of these constructions was to provide a method of
producing infinite ``self-similar'' interval exchanges which could
easily be seen to be ergodic, and then consider the (ergodic) vertical
flow on translation surfaces of infinite area which arise as their
suspensions.  We note here that while the theory of compact
translation surfaces has seen tremendous progress in the last two
decades, non-compact translation surfaces are not as well understood,
and so a method for producing infinite-area translation surfaces with
an ergodic vertical flow seems to be an interesting problem.

Section~\ref{sec:Background} below contains a quick summary of the
pertinent background.  After very briefly mentioning some classical
ergodic theory of a measure-preserving transformation on a probability
space, we mention non-singular transformations and ergodic theory for
a space of infinite measure.  We then give a quick introduction to
interval exchange transformations and translation surfaces.
References to more thorough presentations are provided.  In
Sections~\ref{sec:ReciprocalTransformations} and
\ref{sec:DiscreteMaharam} we introduce our two constructions and prove
some basic theorems about their dynamics.  Section~\ref{sec:Future}
discusses ongoing work related to the Krieger type of the reciprocal
transformations.

\section{Background}
\label{sec:Background}
For the convenience of the reader, in this section we give some very
basic background from ergodic theory, as well as the theory of
interval exchanges and translation surfaces.  For more details on
standard ergodic theory, see \cite{Petersen} and \cite{Walters}.
Information on infinite ergodic theory can be found in
\cite{Aaronson}, while \cite{DanilenkoSilva} and the references
therein contain information about non-singular ergodic theory.
Finally, introductions to interval exchange transformations and
translation surfaces can be found in \cite{Viana}, \cite{Vorobets},
\cite{Wright}, and \cite{Zorich}.

\subsection{Standard ergodic theory}
Let $(X, \mathcal{B}, \mu)$ be a probability space.  We say that a map
$F: X \to X$ is a \emph{measure-preserving transformation (with
respect to $\mu$)} if $F$ is a measurable bijection and for each
measurable subset $E$ we have $\mu(E) = \mu(F(E))$.  All definitions
throughout this article will be up to a set of measure zero.  That is,
if $F$ was measurable and not a bijection from $X$ to itself, but
if there existed a null set $N$ such that the restriction of $F$ to
$X\setminus N$ is a bijection, we will still refer to $F$ as a
``measurable bijection.''

One of the most fundamental properties of measure-preserving
transformations is given by the famous Poincar\'e recurrence theorem:

\begin{theorem*}[Poincar\'e recurrence theorem]
  Let $F$ be a measure-preserving transformation of a probability
  space $(X, \mathcal{B}, \mu)$.  For every set $E$ of positive
  measure and for almost every $x \in E$, there exists an $n$ such
  that $F^n(x) \in E$.
\end{theorem*}

In language to be described in the next subsection below, the
Poincar\'e recurrence theorem states that every measure-preserving
transformation of a probability space is ``conservative.''

We say that a measurable set $E$ is \emph{$F$-invariant} if $F(E) = E$
(again, we allow this statement to hold up to a null set).  The
transformation $F$ is said to be \emph{ergodic} if the only invariant
sets have measure zero or measure one.

\subsection{Non-singular ergodic theory}
As above, continue to suppose that $(X, \mathcal{B}, \mu)$ is a
probability space and that $F : X \to X$ is a measurable bijection.
We say that $F$ is a \emph{non-singular transformation (with respect
to $\mu$)} if $F$ preserves the $\sigma$-ideal of $\mu$-null sets: i.e.,
$\mu \ll F_* \mu \ll \mu$.  That is, given a measurable set $E$,
$\mu(E) = 0$ if and only if $\mu(F(E)) = 0$ as well; $F$ may distort
the measure of a set, but it must respect whether or not a set has
positive or zero measure.

Unlike measure-preserving transformations, non-singular
transformations need not satisfy the conclusion of the Poincar\'e
recurrence theorem.  To be more precise, we must define ``wandering
sets,'' and the ``conservative'' and ``dissipative'' parts of the
transformation.

A set $W$ is said to be \emph{wandering} under $F$ if it is disjoint
from all of its images: $W \cap F^n(W) = \emptyset$.  The union of all
wandering sets is called the \emph{dissipative part} of $F$ and is
denoted $\mathfrak{D}$.  The complement of $\mathfrak{D}$ is the
\emph{conservative part} of $F$ and is denoted $\mathfrak{C}$.  If
$\mathfrak{C}$ has full measure, then $F$ is called
\emph{conservative}.  If $\mathfrak{D}$ has full measure, then $F$ is
called \emph{totally dissipative}.  Notice that both $\mathfrak{C}$
and $\mathfrak{D}$ are $F$-invariant sets.  The decomposition of $X$
into the conservative and dissipative parts of $F$ is known as the
\emph{Hopf decomposition} of $F$.

The notion of ergodicity makes sense for non-singular transformations
and is defined just as in the measure-preserving case: we say a
non-singular transformation $F$ is \emph{ergodic (with respect to
$\mu$)} if each $F$-invariant set has full measure or zero measure.
As shown in \cite[Proposition~1.2.1]{Aaronson}, if $F$ is
a non-singular bijection and $\mu$ is non-atomic, then $F$ must be
conservative to be ergodic.  Throughout this article, all measures
considered will be assumed to be non-atomic, and so conservativity
will be a prerequisite for ergodicity.

\subsection{Infinite ergodic theory}
The notion of a measure-preserving transformation easily extends to
measure spaces of infinite measure, but the definition of ergodicity
must be rephrased slightly.  Instead of simply requiring that an
invariant set have full (infinite) measure, we should require that the
measure of an invariant set or its complement is zero.  We note too
that measure-preserving transformations of infinite measure spaces
need not be conservative: the conclusion of the Poincar\'e recurrence
theorem does not necessarily hold.  Consider, for example, the map $Fx
= x + 1$ defined on the real line which preserves Lebesgue
measure, but is totally dissipative.

\subsection{Interval exchanges and translation surfaces}
An \emph{interval exchange transformation} is a bijection from an
interval to itself which is a piecewise translation with finitely-many
intervals of continuity.  Each interval exchange is determined by a
permutation describing how the intervals of continuity are exchanged
by the transformation, together with a vector of length data giving
the length of each interval of continuity.  These transformations are
often described pictorially by giving an image where each interval of
continuity appears as a subsegment of the interval, and the image of
each interval appears on an interval below.  For example,
Figure~\ref{fig:iet} describes an interval exchange on the interval $X
= [0, 1)$ where the intervals of continuity are $[0, 0.3)$, $[0.3,
0.8)$, and $[0.8, 1)$ which are denoted $A$, $B$, and $C$.  The
transformation $T : X \to X$ indicated in Figure~\ref{fig:iet}
reverses the order of the intervals so that $T(A) = [0.7, 1)$, $T(B) =
[0.2, 0.7)$,
and $T(C) = [0, 0.2)$.

\begin{figure}[h!]
  \centering
    \begin{tikzpicture}[xscale=7]

      \foreach \x in {0, 0.3, 0.8, 1} {
        \draw (\x, 0.3) -- (\x, 0.7);
      }

      \foreach \x in {0, 0.2, 0.7, 1} {
        \draw (\x, -0.3) -- (\x, -0.7);
      }

      \draw[ultra thick, black!25!white] (0, 0.5) -- node[above] {$A$} (0.3, 0.5);
      \draw[ultra thick, black!75!white] (0.3, 0.5) -- node[above] {$B$} (0.8, 0.5);
      \draw[ultra thick, black!50!white] (0.8, 0.5) -- node[above] {$C$} (1, 0.5);
      
      \draw[ultra thick, black!25!white] (0.7, -0.5) -- node[below] {$A$} (1, -0.5);
      \draw[ultra thick, black!75!white] (0.2, -0.5) -- node[below] {$B$} (0.7, -0.5);
      \draw[ultra thick, black!50!white] (0, -0.5) -- node[below] {$C$} (0.2, -0.5);
      
      

      

      
  \end{tikzpicture}
    \caption{An example of an interval exchange with three intervals
      of continuity.}
    \label{fig:iet}
\end{figure}

By considering certain suspensions of an interval exchange, we can
obtain a \emph{translation surface}, which is a surface endowed with
singular flat metric with trivial holonomy.  These surfaces have a
well-defined notion of direction at each point, aside from a finite
amount of ambiguity at the singularities.  That is, the notion of
``vertical or ``horizontal'' or ``with slope $m$'' makes sense at
all but finitely-many points on the surface, and this allows us to
have a well-defined directional flow on the surface (minus the measure
zero subset consisting of points where the flow would hit a
singularity).  The interval on which the interval exchange is defined
then appears as a geodesic segment on the surface, and the
first-return map of the vertical flow on the surface to that geodesic
segment recovers the original interval exchange.

To be precise about the construction of a translation surface as a
suspension over an interval exchange, we consider a complex vector
$\zeta$ with dimension equal to the number of intervals of continuity.
Suppose the $n$-th component of $\zeta$
has real part equal to the width of the $n$-th interval of continuity
of the interval exchange.  Now consider a broken line emanating from
the left-hand endpoint of the interval consisting of a concatenation
of segments described by $\zeta$.  We then consider the corresponding
broken line given by permuting the entries of $\zeta$ as described
by the interval exchange.  Identifying the segments corresponding to
the same entry of $\zeta$ gives the resulting interval exchange.  In
Figure~\ref{fig:translation_surface} we see the surface obtained by
the construction just described using the interval exchange denoted in
Figure~\ref{fig:iet} and the vector $\zeta = (0.3 + i, 0.5 + 0.2i, 0.2
- i)$.

\begin{figure}[h!]
  \centering
  \begin{tikzpicture}
    \begin{scope}[xscale=7]

      \draw[fill=black!25!white,draw=black!25!white] (0, 0) -- (0.3, 0) -- (0.3, 0.8) --
      (0.24, 0.8) -- cycle;
      \draw[fill=black!25!white,draw=black!25!white] (0.7, -0.8) -- (0.7,
      0) -- (0.94, 0) -- cycle;

      \draw[fill=black!75!white,draw=black!75!white] (0.3, 0) -- (0.3,
      1) -- (0.8, 1.2) -- (0.8, 0) -- cycle;
      \draw[fill=black!75!white,draw=black!75!white] (0.2, -1) -- (0.2,
      0) -- (0.7, 0) -- (0.7, -0.8) -- cycle;

      \draw[fill=black!50!white,draw=black!50!white] (0.8, 0) -- (0.8,
      1.2) -- (1, 0.2) -- (0.94, 0) -- cycle;
      \draw[fill=black!50!white,draw=black!50!white] (0.24, 0.8) --
      (0.3, 1) -- (0.3, 0.8) -- cycle;
      \draw[fill=black!50!white,draw=black!50!white] (0, 0) -- (0.2,
      -1) -- (0.2, 0) -- cycle;
      
      \draw[black,thick] (0, 0) -- (0.94, 0);
      \draw[black,thick] (0.24, 0.8) -- (0.3, 0.8);
      
      \draw[ultra thick, black!25!white] (0, 0) -- node[above] {$A$} (0.3, 1);
      \draw[ultra thick, black!75!white] (0.3, 1) -- node[above] {$B$} (0.8,
      1.2);
      \draw[ultra thick, black!50!white] (0.8, 1.2) -- node[above] {$C$} (1,
      0.2);
      \draw[ultra thick, black!50!white] (0, 0) -- node[below] {$C$} (0.2, -1);
      \draw[ultra thick, black!75!white] (0.2, -1) -- node[below] {$B$} (0.7,
      -0.8);
      \draw[ultra thick, black!25!white] (0.7, -0.8) -- node[below] {$A$} (1,
      0.2);

    \end{scope}
  \end{tikzpicture}
  \caption{A translation surface obtained as a suspension over the
    interval exchange described in Figure~\ref{fig:iet}.}%
  \label{fig:translation_surface}
\end{figure}

In addition to interval exchanges which preserve the Lebesgue measure
of an interval, we will also be concerned with \emph{affine interval
exchanges} which are non-singular transformations of an interval given
by a piecewise affine bijection of the interval to itself.  Just as
interval exchanges can be denoted pictorially, affine interval
exchanges can be denoted in a similar way, such as in
Figure~\ref{fig:affiet_example}.
  
\begin{figure}[h!]
  \centering
    \begin{tikzpicture}[xscale=7]

      \foreach \x in {0, 0.3, 0.8, 1} {
        \draw (\x, 0.3) -- (\x, 0.7);
      }

      \foreach \x in {0, 0.1, 0.5, 1} {
        \draw (\x, -0.3) -- (\x, -0.7);
      }

      \draw[ultra thick, black!25!white] (0, 0.5) -- node[above] {$A$} (0.3, 0.5);
      \draw[ultra thick, black!75!white] (0.3, 0.5) -- node[above] {$B$} (0.8, 0.5);
      \draw[ultra thick, black!50!white] (0.8, 0.5) -- node[above] {$C$} (1, 0.5);
      
      \draw[ultra thick, black!25!white] (0.5, -0.5) -- node[below] {$A$} (1, -0.5);
      \draw[ultra thick, black!75!white] (0.1, -0.5) -- node[below] {$B$} (0.5, -0.5);
      \draw[ultra thick, black!50!white] (0, -0.5) -- node[below] {$C$} (0.1, -0.5);
      
      

      

      
  \end{tikzpicture}
    \caption{An example of an affine interval exchange with three intervals
      of continuity.  The intervals are permuted just as in the
      interval exchange in Figure~\ref{fig:iet}, but the lengths of
      the intervals can be distorted.}
    \label{fig:affiet_example}
\end{figure}

\section{Reciprocal Transformations}
\label{sec:ReciprocalTransformations}
\subsection{Introduction and an example}
Let $(X, \mathcal{B}, \mu)$ be a probability space.  We call a
measurable bijection $\Phi : X \to X$ a \emph{scaling involution} if
$\Phi^2 = \mathrm{id}$ and there exists a set $S$ of positive measure
such that the following conditions are satisfied:
\begin{itemize}
\item $X = S \sqcup \Phi(S)$
\item $\mu(S) < \mu(\Phi(S))$ ($S$ is the smaller set)
\item For each measurable $E \subseteq S$, $\mu(\Phi(E)) =
  \frac{\mu(\Phi(S))}{\mu(S)} \mu(E)$.
\end{itemize}
That is, $\Phi$ simply exchanges $S$ and $\Phi(S)$, uniformly scaling
the measure of subsets in $S$ or $\Phi(S)$.  Notice that since $\Phi$
is an involution, the third condition above implies that for each
measurable $E \subseteq \Phi(S)$ we have $\mu(\Phi(E)) =
\frac{\mu(S)}{\mu(\Phi(S))} \mu(E)$.

For notational convenience, we let $\rho$ denote the ratio
$\frac{\mu(\Phi(S))}{\mu(S)}$.  Observe that sets $E \subseteq
\Phi(S)$ are scaled by a factor of $\rho^{-1}$ under the
application of $F$.  The last condition in the definition of a
scaling involution can also be described as requiring the
Radon-Nikodym derivative $\frac{dF_*\mu}{d\mu}$ equal $\rho$ at each
point of $S$ and $\rho^{-1}$ at each point of $\Phi(S)$.

\begin{example}
  \label{ex:scaling_third}
  If $X = [0, 1)$ equipped with the Lebesgue measure $\mu$, then the
    affine map exchanging $S = [0, \sfrac{1}{3})$ and $\Phi(S) =
      [\sfrac{1}{3}, 1)$ is a scaling involution with ratio $\rho =
        2$.
\end{example}

By a \emph{reciprocal transformation} we mean the composition $F =
\Phi T$ where $T$ is a measure-preserving transformation and $\Phi$ is
a scaling involution on a probability space.

\begin{remark}
  We observe that in defining the reciprocal transformation we made a
  choice to consider the composition $F = \Phi T$, but could have
  instead considered the other composition, $G = T \Phi$.  We
  observe, though, that these maps are conjugate to one another with
  $\Phi$ being the conjugating map:
  \[
  \Phi F = \Phi (\Phi T) = T = (T \Phi) \Phi = G \Phi.
  \]
  Thus dynamical properties of $F$ such conservativity and ergodicity
  hold if and only if $G$ also has these properties.  Throughout this
  article we will focus on the composition $F = \Phi T$ for concreteness.
\end{remark}

\begin{example}
  \label{ex:pair_rotation}
  Consider the
  case when $X = [0, 1)$ is an interval, and $T : X \to X$ is the
  interval transformation denoted in Figure~\ref{fig:iet_example}.
  This transformation is a piecewise translation with four intervals
  of continuity denoted $A$, $B$, $C$, and $D$.  The transformation
  simply acts as a pair of circle rotations, rotating $A \cup B$ by
  the length of $B$, and rotating $C \cup D$ by the length of $D$.
  \begin{figure}[h!]
    \centering
    \begin{tikzpicture}[xscale=7]

      \foreach \x in {0, 0.333, 1} {
        \draw (\x, 0.3) -- (\x, 0.7);
      }

      \foreach \x in {0, 0.333, 1} {
        \draw (\x, -0.3) -- (\x, -0.7);
      }
      
      \foreach \x in {0, 0.242, 1} {
        \draw (\x, 0.3) -- (\x, 0.7);
      }
      
      \foreach \x in {0, 0.091, 1} {
        \draw (\x, -0.3) -- (\x, -0.7);
      }

      \foreach \x in {0, 0.707, 1} {
        \draw (\x, 0.3) -- (\x, 0.7);
      }
      
      \foreach \x in {0, 0.626, 1} {
        \draw (\x, -0.3) -- (\x, -0.7);
      }

      \draw[ultra thick, black!25!white] (0, 0.5) -- node[above] {$A$} (0.242, 0.5);
      \draw[ultra thick, black!75!white] (0.242, 0.5) -- node[above] {$B$} (0.333, 0.5);
      \draw[ultra thick, black!50!white] (0.333, 0.5) -- node[above] {$C$} (0.707, 0.5);
      \draw[ultra thick, black] (0.707, 0.5) -- node[above] {$D$} (1, 0.5);
      
      \draw[ultra thick, black!75!white] (0, -0.5) -- node[below] {$B$} (0.091, -0.5);
      \draw[ultra thick, black!25!white] (0.091, -0.5) -- node[below] {$A$} (0.333, -0.5);
      \draw[ultra thick, black] (0.333, -0.5) -- node[below] {$D$} (0.626, -0.5);
      \draw[ultra thick, black!50!white] (0.626, -0.5) -- node[below] {$C$} (1, -0.5);
    \end{tikzpicture}
    \caption{An interval exchange transformation $T$ with
      four intervals permuted by $(A, B)(C, D)$.}
    \label{fig:iet_example}
  \end{figure}
  Now take $\Phi$ to be the affine piecewise involution which
  exchanges $S = A \cup B$ and $\Phi(S) = C \cup D$, described in
  Figure~\ref{fig:phi_example}.
  \begin{figure}[h!]
    \centering
    \begin{tikzpicture}[xscale=7]

      \foreach \x in {0, 0.333, 1} {
        \draw (\x, 0.3) -- (\x, 0.7);
      }
      
      \foreach \x in {0, 0.333, 1} {
        \draw (\x, -0.3) -- (\x, -0.7);
      }
      
      \draw[ultra thick, black] (0, 0.5) -- node[above] {$S$} (0.333, 0.5);
      \draw[ultra thick, black!50!white] (0.333, 0.5) -- node[above] {$\Phi S$} (1, 0.5);
      \draw[ultra thick, black!50!white] (0, -0.5) -- node[below] {$\Phi S$} (0.333, -0.5);
      \draw[ultra thick, black] (0.333, -0.5) -- node[below] {$S$} (1, -0.5);
    \end{tikzpicture}
    \caption{A scaling involution $\Phi$ exchanging $S$ and $\Phi(S)$.}
    \label{fig:phi_example}
  \end{figure}
  The resulting composition $F = \Phi T$ is the affine interval
  exchange indicated in Figure~\ref{fig:aiet_example}.
  \begin{figure}[h!]
    \centering
    \begin{tikzpicture}[xscale=7]

    \foreach \x in {0, 0.242, 0.333, 0.707, 1} {
      \draw (\x, 0.3) -- (\x, 0.7);
    }
    
    \foreach \x in {0, 0.146, 0.333, 0.516, 1} {
      \draw (\x, -0.3) -- (\x, -0.7);
    }

    \draw[ultra thick, black!25!white] (0, 0.5) -- node[above] {$A$} (0.242, 0.5);
    \draw[ultra thick, black!75!white] (0.242, 0.5) -- node[above] {$B$} (0.333, 0.5);
    \draw[ultra thick, black!50!white] (0.333, 0.5) -- node[above] {$C$}
    (0.707, 0.5);
    \draw[ultra thick, black] (0.707, 0.5) -- node[above] {$D$}
    (1, 0.5);
    \draw[ultra thick, black!25!white] (0.516, -0.5) -- node[below] {$A$} (1, -0.5);
    \draw[ultra thick, black!75!white] (0.333, -0.5) -- node[below] {$B$} (0.516, -0.5);

    \draw[ultra thick, black!50!white] (0.146, -0.5) -- node[below] {$C$} (0.333, -0.5);
    \draw[ultra thick, black] (0, -0.5) -- node[below] {$D$} (0.146,
    -0.5);
    \end{tikzpicture}
    \caption{An affine interval exchange $F$ obtained by composing
      an interval exchange $T$ with a scaling involution $\Phi$.}
    \label{fig:aiet_example}
  \end{figure}
\end{example}

This simple construction gives us a way to build non-singular
transformations where we can use information about the dynamics of the
measure-preserving $T$ to glean information about the dynamics of $F$.
One of our primary goals is to determine what conditions on $\Phi$ and
$T$ are sufficient to ensure that the reciprocal transformation $F = \Phi T$ is
ergodic.

\subsection{Conservativity and the first-return to $S$}
As we will be interested in the ergodicity of reciprocal
transformations and conservativity is prerequisite for ergodicity when
$\mu$ is non-atomic, we must first understand when a reciprocal
transformation is conservative.  To this end it will be helpful to
study when the first-return map to $S$ (the smaller of the two sets
exchanged by the scaling involution $\Phi$) is defined.  Observe that
we \emph{are not} currently supposing the reciprocal transformation $F
= \Phi T$ is conservative, and so \emph{a priori} the first-return map
to an arbitrary subset need not be defined.  As we will see, however,
the first-return map of a reciprocal transformation to $S$ \emph{is}
always defined.

We first observe that when $F$ is applied, any set whose image is in
$S$ will have its measure shrink by a factor of $\rho^{-1}$, and any
set whose image is in $\Phi(S)$ will have its measure grow by a factor
of $\rho$.  This is illustrated in Example~\ref{ex:wanderingFS} below.

\begin{example}
  \label{ex:wanderingFS}
  Consider the interval exchange $T$ indicated in
  Figure~\ref{fig:wanderingIET} where each interval has length $1/9$
  or $2/9$.
  \begin{figure}[h!]
    \centering
    \begin{tikzpicture}[xscale=7]
      \foreach \x in {0, 0.1111, 0.3333, 0.7777, 1} {
        \draw (\x, 0.3) -- (\x, 0.7);
      }

      \foreach \x in {0, 0.2222, 0.3333, 0.7777, 1} {
        \draw (\x, -0.3) -- (\x, -0.7);
      }

      \draw[ultra thick, black!25!white] (0, 0.5) -- node[above] {$A$} (0.1111, 0.5);
      \draw[ultra thick, black!75!white] (0.1111, 0.5) -- node[above] {$B$} (0.3333, 0.5);
      \draw[ultra thick, black] (0.3333, 0.5) -- node[above] {$C$} (0.777, 0.5);
      \draw[ultra thick, black!50!white] (0.7777, 0.5) -- node[above] {$D$} (1, 0.5);

      \draw[ultra thick, black!50!white] (0, -0.5) -- node[below] {$D$} (0.2222, -0.5);
      \draw[ultra thick, black!25!white] (0.2222, -0.5) -- node[below] {$A$} (0.3333, -0.5);
      \draw[ultra thick, black] (0.3333, -0.5) -- node[below] {$C$} (0.7777, -0.5);
      \draw[ultra thick, black!75!white] (0.7777, -0.5) -- node[below] {$B$} (1, -0.5);
    \end{tikzpicture}
    \caption{An interval exchange transformation.}
    \label{fig:wanderingIET}
  \end{figure}
  Composing this map with the $\Phi$ exchanging $[0, 1/3)$ and $[1/3,
      1)$ gives us the affine interval exchange indicated $F$
      indicated in Figure~\ref{fig:composition}.
  \begin{figure}[h!]
    \centering
    \begin{tikzpicture}[xscale=7]
      \foreach \x in {0, 0.1111, 0.3333, 0.7777, 1} {
        \draw (\x, 0.3) -- (\x, 0.7);
      }

      \foreach \x in {0, 0.2222, 0.3333, 0.7777, 1} {
        \draw (\x, -0.3) -- (\x, -0.7);
      }

      \draw[ultra thick, black!25!white] (0, 0.5) -- node[above] {$A$} (0.1111, 0.5);
      \draw[ultra thick, black!75!white] (0.1111, 0.5) -- node[above] {$B$} (0.3333, 0.5);
      \draw[ultra thick, black] (0.3333, 0.5) -- node[above] {$C$} (0.7777, 0.5);
      \draw[ultra thick, black!50!white] (0.7777, 0.5) -- node[above] {$D$} (1, 0.5);

      \draw[ultra thick, black] (0, -0.5) -- node[below] {$C$} (0.2222, -0.5);
      \draw[ultra thick, black!75!white] (0.2222, -0.5) -- node[below] {$B$} (0.3333, -0.5);
      \draw[ultra thick, black!50!white] (0.3333, -0.5) -- node[below] {$D$} (0.7777, -0.5);
      \draw[ultra thick, black!25!white] (0.7777, -0.5) -- node[below] {$A$} (1, -0.5);
    \end{tikzpicture}
    \caption{A reciprocal transformation.}
    \label{fig:composition}
  \end{figure}
  Observe that the intervals labeled $B$, $C$, and $D$ have image in
  $S$ and their measures each shrink by a factor of $1/2$.  The
  intervals labeled $A$ and $E$, however, have images in $\Phi(S)$ and
  their measures grow by a factor of $2$.

  Now consider the first return map to $S = [0, 1/3)$ which we will
  denoted $F_S$ and seen in Figure~\ref{fig:wandering_FS}.
  \begin{figure}[h!]
    \centering
    \begin{tikzpicture}[xscale=7]
      \foreach \x in {0, 0.1111, 0.3333} {
        \draw (\x, 0.3) -- (\x, 0.7);
      }

      \foreach \x in {0, 0.2222, 0.3333} {
        \draw (\x, -0.3) -- (\x, -0.7);
      }

      \draw[ultra thick, black!25!white] (0, 0.5) -- node[above] {$A$} (0.1111, 0.5);
      \draw[ultra thick, black!75!white] (0.1111, 0.5) -- node[above] {$B$} (0.3333, 0.5);

      \draw[ultra thick, black!25!white] (0, -0.5) -- node[below] {$A$} (0.2222, -0.5);
      \draw[ultra thick, black!75!white] (0.2222, -0.5) -- node[below] {$B$} (0.3333, -0.5);
    \end{tikzpicture}
    \caption{The first return map admits a wandering interval.}
    \label{fig:wandering_FS}
  \end{figure}
  We observe that the interval $B \setminus F_S(B)$ is wandering under
  $F_S$, always shrinking and shifting to the right.  We also observe
  that the interval $A$ grows by a factor of $2$ under $F_S$,
  corresponding to the fact $F_S(A) = F^3(A)$; the image of $A$ grows
  by a factor of $2$ under the first two applications of $F$ but then
  $F^2(A)$ shrinks by a factor of $1/2$ when $F$ is applied again.
\end{example}


Noting that the measure of the image of a set must grow by $\rho$ if
that image is not in $S$ is the key observation required for showing
almost every point must eventually enter $S$.

\begin{lemma}
  \label{lemma:enterS}
  Let $F = \Phi T$ be the reciprocal transformation given by composing
  a measure-preserving bijection $T$ with a scaling involution
  $\Phi$ of scaling ratio $\rho > 1$.  Then for $\mu$-almost every $x
  \in X$, there exists an $n > 0$ so that $F^nx \in S$.
\end{lemma}
\begin{proof}
  Let $U$ be the set of points in $X$ which never enter $S$:
  \[
  U = \left\{x \in X \, \big| \, F^nx \notin S \text{ for all } n > 0\right\}.
  \]
  We claim $U$ must have measure zero.  If not, then all iterates of
  $U$ remain in $\Phi(S)$.  However, this means every measurable
  subset of $U$ has its size increased by a factor of $\rho$ with each
  application of $F$.  Since $\Phi(S)$ has finite measure, though,
  this is impossible unless $\mu(U) = 0$.
\end{proof}

We emphasize again that throughout the paper all maps are assumed to
be measurably bijective unless otherwise explicitly stated.  Some
results may hold even for transformations which are not bijective,
although some proofs given here would need to be modified as we often
consider the images instead of preimages in the proofs.

Lemma~\ref{lemma:enterS} implies, in particular, that there is a
well-defined first-return map from $S$ to itself which we denote $F_S
: S \to S$.  Notice that dynamical properties of $F$ are reflected in
$F_S$ and vice versa when $F_S$ is also a non-singular bijection.
Notice, though, that $F_S$ may fail to be surjective in general, and
in such situations is not non-singular.

\begin{example}
  \label{ex:nonsurjective}
  Consider the reciprocal transformation formed by composing the
  interval exchange transformation $T$ consisting of four intervals, $A
  = [0, 1/6)$, $B = [1/6, 1/3)$, $C = [1/3, 1/2)$, and $D = [1/2, 1)$
  which are permuted as shown in Figure~\ref{fig:nonsurjective_IET}.  
  Composing $T$ with the scaling involution of
  Example~\ref{ex:scaling_third}, resulting in the transformation
  shown in Figure~\ref{fig:nonsurjective_F}.

  \begin{figure}[h!]
    \centering
    \begin{tikzpicture}[xscale=7]
      \foreach \x in {0, 0.1667, 0.3333, 0.5, 1} {
        \draw (\x, 0.3) -- (\x, 0.7);
      }

      \foreach \x in {0, 0.1667, 0.3333, 0.8333, 1} {
        \draw (\x, -0.3) -- (\x, -0.7);
      }

      \draw[ultra thick, black!25!white] (0, 0.5) -- node[above] {$A$} (0.1667,
      0.5);
      \draw[ultra thick, black!75!white] (0.1667, 0.5) -- node[above] {$B$} (0.3333,
      0.5);
      \draw[ultra thick, black!50!white] (0.3333, 0.5) -- node[above] {$C$} (0.5,
      0.5);
      \draw[ultra thick, black] (0.5, 0.5) -- node[above] {$D$} (1, 0.5);

      \draw[ultra thick, black!50!white] (0, -0.5) -- node[below] {$C$} (0.1667,
      -0.5);
      \draw[ultra thick, black!25!white] (0.1667, -0.5) -- node[below] {$A$} (0.3333,
      -0.5);
      \draw[ultra thick, black] (0.3333, -0.5) -- node[below] {$D$} (0.8333, -0.5);
      \draw[ultra thick, black!75!white] (0.8333, -0.5) -- node[below] {$B$} (1, -0.5);
    \end{tikzpicture}
    \caption{An interval exchange $T$ defined on four subintervals of
      $[0, 1)$.}
    \label{fig:nonsurjective_IET}
  \end{figure}

  \begin{figure}[h!]
    \centering
    \begin{tikzpicture}[xscale=7]
      \foreach \x in {0, 0.1667, 0.3333, 0.5, 1} {
        \draw (\x, 0.3) -- (\x, 0.7);
      }

      \foreach \x in {0, 0.25, 0.3333, 0.6666, 1} {
        \draw (\x, -0.3) -- (\x, -0.7);
      }

      \draw[ultra thick, black!25!white] (0, 0.5) -- node[above] {$A$} (0.1667,
      0.5);
      \draw[ultra thick, black!75!white] (0.1667, 0.5) -- node[above] {$B$} (0.3333,
      0.5);
      \draw[ultra thick, black!50!white] (0.3333, 0.5) -- node[above] {$C$} (0.5,
      0.5);
      \draw[ultra thick, black] (0.5, 0.5) -- node[above] {$D$} (1, 0.5);

      \draw[ultra thick, black] (0, -0.5) -- node[below] {$D$} (0.25,
      -0.5);
      \draw[ultra thick, black!75!white] (0.25, -0.5) -- node[below] {$B$} (0.3333,
      -0.5);
      \draw[ultra thick, black!50!white] (0.3333, -0.5) -- node[below] {$C$} (0.6666, -0.5);
      \draw[ultra thick, black!25!white] (0.6666, -0.5) -- node[below] {$A$} (1, -0.5);
    \end{tikzpicture}
    \caption{The reciprocal transformation $F$ obtained by composing
      the map $T$ from Figure~\ref{fig:nonsurjective_IET} with the
      involution $\Phi$ exchanging $[0, 1/3)$ and $[1/3, 1)$.}
    \label{fig:nonsurjective_F}
  \end{figure}

  The first return map to $S = [0, 1/3)$ is then easily seen to be the
  map in Figure~\ref{fig:nonsurjective_FS}.  The dashed interval in
  Figure~\ref{fig:nonsurjective_FS} indicating that $[0, 1/12)$ is
  not in the image of $F_S$ and so $F_S^{-1}([0, 1/12)) =
  \emptyset$ and thus $F_S$ in this example is not non-singular.

  \begin{figure}[h!]
    \centering
    \begin{tikzpicture}[xscale=7]
      \foreach \x in {0, 0.1667, 0.3333} {
        \draw (\x, 0.3) -- (\x, 0.7);
      }
      \foreach \x in {0, 0.0833, 0.25, 0.333} {
        \draw (\x, -0.3) -- (\x, -0.7);
      }
      \draw[ultra thick, black!25!white] (0, 0.5) -- node[above] {$A$} (0.1667,
      0.5);
      \draw[ultra thick, black!75!white] (0.1667, 0.5) -- node[above] {$B$}
      (0.3333, 0.5);
      \draw[ultra thick, black!25!white] (0.08333, -0.5) -- node[below] {$A$}
      (0.25, -0.5);
      \draw[ultra thick, black!75!white] (0.25, -0.5) -- node[below] {$B$}
      (0.3333, -0.5);
      \draw[dashed] (0, -0.5) -- (0.0833, -0.5);
    \end{tikzpicture}
    \caption{The first return map $F_S$ of the map $F$ from
      Figure~\ref{fig:nonsurjective_F} to $S = [0, 1/3)$.  The dashed
        segment
      is not in the image of $F_S$.}
    \label{fig:nonsurjective_FS}
  \end{figure}
\end{example}

\begin{lemma}
  \label{lemma:FS_surjective_nonsingular}
  If the first-return map $F_S$ is surjective, then it is non-singular.
\end{lemma}
\begin{proof}
  First note that as $F$ is a bijection, the first-return map $F_S$ is
  always injective: suppose $s, s' \in S$ with $F_S(s) = F_S(s')$
  where $F_S(s) = F^n(s)$ and $F_S(s') = F^{n'}(s')$ where $n, n'$ are
  the smallest positive integers so that $F^{n}(s), F^{n'}(s') \in S$.
  If $n = n'$ then since $F$ is injective, we have $s = s'$.
  Otherwise suppose without loss of generality that $n > n'$.
  Applying $F^{-n'}$ to both sides of the equation $F^n(s) =
  F^{n'}(s')$ yields $F^{n - n'}(s) = s'$, but as $n - n' < n$ this
  contradicts the definition of $n$.

  If $F_S$ is surjective, it is thus bijective.  Hence $F$ is
  non-singular if and only if for each $E \subseteq S$ of positive
  measure we have that $F(E)$ has positive measure, and this is clear
  from the decomposition $S = \bigsqcup_{n \in \mathbb{N}} S_n$ with
  $\frac{dF_{S*}}{d\mu}$ equaling $\rho^{n-2}$ on $S_n$.
\end{proof}

\begin{proposition}
  \label{prop:F_conserverative_iff_FS}
  The reciprocal transformation $F$ is conservative if and only if the
  first-return map $F_S$ is conservative.
\end{proposition}
\begin{proof}
  Suppose $W$ was an $F$-wandering set.  By Lemma~\ref{lemma:enterS},
  almost every point of $W$ eventually enters $S$.  Let $W_n$ be the
  subset of $W$ that first enters $S$ after $n$ applications of $F$,
  for $n \geq 0$.  If $W$ has positive measure, then at least one
  $W_n$ does as well, and $F^n(W_n)$ is an $F_S$-wandering set of
  positive measure.  So, non-conservativity of $F$ implies
  non-conservativity of $F_S$.

  Now suppose $W \subseteq S$ is an $F_S$-wandering set.  Then $W$
  must also be an $F$-wandering set, for if $F^m(W) \cap F^n(W) \neq
  \emptyset$, then by assumption $F^m(W) \cap F^n(W) \subseteq
  \Phi(S)$.  There exists some subset $U$ of $F^m(W) \cap F^n(W)$ such
  that $F^k(U) \subseteq S$ for some $k > 0$ and $F^j(U) \cap S =
  \emptyset$ for all $0 < j < k$.  But then $U' = F^{-(n+k)}(U)$ and
  $U'' = F^{-(m+k)}(U)$ are subsets of $W \subseteq S$ with the property
  that $F^{n+k}(U') \cap F^{m+k}(U'')$ is a non-empty subset of $S$,
  and there exists $n'$ and $m'$ so that $F_S^{n'}(U') \cap
  F_S^{m'}(U'')$ is non-empty, contradicting that $W$ was wandering.
  Thus $F_S$-wandering sets are also $F$-wandering sets.  Since $F$
  and $F_S$ are non-singular, non-conservativity of $F_S$ implies
  non-conservativity of $F$.
\end{proof}

To study
the dynamics of $F_S$, it is helpful to notice that $S$ naturally
partitions into a countable collection of subsets.  For each $n \geq
1$, we let $S_n$ denote the subset of $S$ which first returns to $S$
after $n$ applications of $F$:
\[
S_n := \left\{ x \in S \, \big| \, F_S(x) = F^n(x) \right\}.
\]
As $F = \Phi T$ where $\Phi$ scales $S$ by $\rho$ and $\Phi(S)$
by $\rho^{-1}$, notice that $\mu(F_S(S_n)) = \rho^{n-2} \mu(S_n)$.
That is, $F_S$ shrinks $S_1$ by $\rho^{-1}$, preserves the measure of
$S_2$, and enlarges all other $S_n$.  As a consequence, if $W$ was an
$F_S$-wandering set, it must eventually enter and remain in $S_1$
under iteration of $F_S$.  This simple observation establishes the
following lemma.

\begin{lemma}
  \label{lemma:conservativeFS}
  If $\mu(S_1) = 0$, then $F_S$ and $F$ are conservative.
\end{lemma}

As we will see, while Lemma~\ref{lemma:conservativeFS} easily
guarantees that a transformation is conservative, it will be too
strong for the applications we have in mind in
Section~\ref{sec:DiscreteMaharam}; see Example~\ref{ex:Ftilde_not_ergodic}.

\subsection{Ergodicity}
We now turn our attention to the ergodicity of reciprocal
transformations.  
First notice that the construction of $F = \Phi T$ as the composition
of a scaling involution and measure-preserving transformation
places strong restrictions on invariant subsets.  As
Proposition~\ref{prop:FinvProportion} shows, the proportion of an
$F$-invariant set which lives in $S$ (respectively, $\Phi(S)$) must
equal the measure of $S$ (resp., $\Phi(S)$).

\begin{proposition}
  \label{prop:FinvProportion}
  If $E$ is an $F$-invariant set, then
  \[
  \frac{\mu(E \cap S)}{\mu(E)} = \mu(S)
  \text{ and }
  \frac{\mu(E \cap \Phi(S))}{\mu(E)} = \mu(\Phi(S))
  \]
\end{proposition}
\begin{proof}
  Let $r$ denote the proportion of $E$ which lives in $S$, $r =
  \frac{\mu(E \cap S)}{\mu(E)}$, and so $1 - r$ is the proportion of
  $E$ which lives in $\Phi(S)$.  Now consider how $\Phi$ distorts the
  measure of $E$:
  \begin{align*}
    \mu(\Phi(E))
    &= \mu(\Phi(\left[E \cap S\right] \cup \left[E \cap
      \Phi(S)\right])) \\
    &= \mu(\Phi(E \cap S)) + \mu(\Phi(E \cap \Phi(S))) \\
    &= \rho \mu(E \cap S) + \rho^{-1}\mu(E \cap \Phi(S)) \\
    &= \rho r \mu(E) + \rho^{-1}(1 - r)\mu(E).
  \end{align*}
  
  Notice that as $E$ is $F$-invariant, so $E = F(E) = \Phi T(E)$,
  and $\Phi$ is an involution we must have $\Phi(E) = T(E)$.  As
  $T$ is measure-preserving, however, this means $\mu(\Phi(E)) =
  \mu(T(E)) = \mu(E)$.  That is, the above equalities become simply
  \[
  \mu(E) = \rho r \mu(E) + \rho^{-1}(1 - r) \mu(E).
  \]
  Dividing by $\mu(E)$ we have $1 = \rho r + \frac{1 - r}{\rho}$ or
  $\rho = \rho^2 r + 1 - r = r(\rho^2 - 1) + 1$.  
  Solving for $r$ yields $r = \frac{\rho - 1}{\rho^{2} - 1}$ or simply
  $r = \frac{1}{\rho + 1}$.  Keeping in mind $\rho = \frac{\mu(\Phi
    S)}{\mu(S)}$ this becomes
  \[
  r = \frac{1}{\frac{\mu(\Phi(S))}{\mu(S)} + 1} =
  \frac{1}{\left(\frac{\mu(\Phi(S)) + \mu(S)}{\mu(S)}\right)} =
  \frac{\mu(S)}{\mu(\Phi(S)) + \mu(S)} = \mu(S).
  \]
  That the proportion of $E$ in $\Phi(S)$ equals $\mu(\Phi(S))$ now easily follows.
\end{proof}

We note that the expression of a reciprocal transformation as a
composition $F = \Phi T$ with $\Phi$ a scaling involution and $T$ a
probability-preserving transformation allows us the following condition
for the ergodicity of $F$.

\begin{proposition}
  \label{prop:EcontainsSergodic}
  The reciprocal transformation $F = \Phi T$ is ergodic if and only if for every
  set $E$ which does not contain $S$ (up to a null set), $\Phi(E) \neq T(E)$.
\end{proposition}
\begin{proof}
  Suppose that $F$ is ergodic and $E$ does not contain all of $S$.
  Then $E$ can not have full measure and is not an invariant set.
  That is, $F(E) \neq E$ which is equivalent to $\Phi T(E) \neq E$,
  and as
  $\Phi$ is an involution this is equivalent to $T(E) \neq \Phi(E)$.

  Suppose now that for every $E$ not containing $S$, $E$ is not
  $F$-invariant.  Thus if a set $U$ were to be $F$-invariant, it must
  contain $S$.  In that case, $U \cap S = S$ and so the conclusion of
  Proposition~\ref{prop:FinvProportion} becomes
  \[
  \frac{\mu(U \cap S)}{\mu(U)} = \frac{\mu(S)}{\mu(U)} = \mu(S)
  \implies \mu(U) = 1.
  \]
\end{proof}

We can now show that determining whether or not $F$ is ergodic can be
rephrased in terms of the ergodicity of $F_S$.  

\begin{theorem}
  \label{thm:F_iff_FS_ergodic}
  The transformation $F = \Phi T$ is ergodic if the first-return map
  to $S$, $F_S$, is ergodic.
\end{theorem}
\begin{proof}
  Suppose $F_S$ is ergodic and let $E \subseteq X$ be an $F$-invariant
  set.  This would imply that $E \cap S$ is
  $F_S$-invariant, and so ergodicity of $F_S$ would force $S \subseteq
  E$.  The contrapositive of Proposition~\ref{prop:EcontainsSergodic}
  then shows $F$ is ergodic.
\end{proof}

\begin{example}
  \label{ex:ergodic_aiet}
  With Theorem~\ref{thm:F_iff_FS_ergodic}, we now have a tool that can
  sometimes be used to easily establish the ergodicity of certain
  non-singular transformations.  In particular, consider affine interval
  exchanges such as those described in Example~\ref{ex:pair_rotation}.
  The maps described in that example give a 3-parameter family of affine
  interval exchanges, corresponding to the lengths of the permuted
  subintervals.  It is easy to see that the first-return map, $F_S$, of
  such a map is itself a circle rotation where points are translated
  (rotated) by $\mu(B) + \rho^{-1} \mu(D)$, where $\mu$ refers to the
  Lebesgue measure.  Thus this $F_S$, and hence $F$, will be ergodic
  when $\mu(B) + \rho^{-1} \mu(D)$ is an irrational number.  Note
  however that $F^2$ \emph{is not} ergodic, as each of the rotated
  circles $A \cup B$ and $C \cup D$ is an invariant subset of $F^2$.
\end{example}

\section{The Discrete Maharam Extension}
\label{sec:DiscreteMaharam}
In \cite{Maharam}, Maharam gave a construction
which associates to a non-singular transformation on a
probability space a measure-preserving transformation on a space of
infinite measure.  We now describe a construction which is similar,
and can be thought of as a rescaling of the Maharam extension, but
restricted to a measure zero subset of the space introduced by the
Maharam extension.  Note that since we are restricting to a measure
zero subset, ergodic theoretic results of Maharam do not instantly
carry over to the setting we are about to describe.

Throughout this section we again consider a probability space $(X,
\mathcal{B}, \mu)$, a measure-preserving transformation $T$, a scaling
involution $\Phi$ with scaling ratio $\rho > 1$, and let $F$ be
the reciprocal transformation $F = \Phi T$.

We will place a particular measure on $X \times \mathbb{Z}$ and
construct a map $\widetilde{F}$ on $X \times \mathbb{Z}$ which
preserves this measure and has $F$ as a factor.

First, let $\pi : X \times \mathbb{Z} \to X$ denote the projection onto
the first factor.  For any $E \subseteq X \times \mathbb{Z}$ and any
$n \in \mathbb{Z}$, let $E_n$ denote the set $E \cap (X \times
\{n\})$.  We will refer to $E_n$ as the \emph{$n$-th level of $E$}.

We define a measure $\widetilde{\mu}$ on $X \times \mathbb{Z}$ as
follows: for each measurable $E \subseteq X \times \mathbb{Z}$, we
declare
\[
\widetilde{\mu}(E) = \sum_{n \in \mathbb{Z}} \rho^{-n} \mu(\pi(E_n))
\]
where $\rho$ is the scaling ratio of $\Phi$.  That is, we imagine $X
\times \mathbb{Z}$ as consisting of copies of $X$ scaled by powers of
$\rho$ so that ``higher'' copies of $X$ are smaller than ``lower''
copies.  Note that $\widetilde{\mu}$ is an infinite, but
$\sigma$-finite, measure on $X \times \mathbb{Z}$.

Now consider the map $\widetilde{F} : X \times \mathbb{Z} \to X
\times \mathbb{Z}$ defined by
\[
\widetilde{F}(x, n) = \begin{cases}
  (F(x), n+1) &\text{ if } x \in T^{-1}(S) \\
  (F(x), n-1) &\text{ if } x \in T^{-1}(\Phi(S))
\end{cases}
\]
Notice $F\pi = \pi\widetilde{F}$.

We will refer to $\widetilde{F} : X \times \mathbb{Z} \to X \times
\mathbb{Z}$ together with the measure $\widetilde{\mu}$ as the
\emph{discrete Maharam extension} of $F : X \to X$ with measure $\mu$.
(The $\sigma$-algebra on $X \times \mathbb{Z}$ is simply the product $\sigma$-algebra
of the $\sigma$-algebra $\mathcal{B}$ on $X$ with the discrete
$\sigma$-algebra, i.e. the power set, on $\mathbb{Z}$.)

\begin{example}
  \label{ex:discrete_maharam_example}
  We now describe the discrete Maharam extension of the reciprocal
  transformation from Example~\ref{ex:pair_rotation}.  Here $X$ is in
  the interval $[0, 1)$, and so the levels $X \times \mathbb{Z}$ are
  scaled versions of the interval; we may think of $X \times \{n\}$
  as $[0, \rho^{-n}) = [0, 2^{-n})$.  The level $X \times \{n\} =
  [0, 2^{-n})$ is partitioned into intervals $A_n$, $B_n$,
  $C_n$, $D_n$, scaled copies of the original $A$, $B$, $C$,
  $D$, subintervals which $X$ was partitioned into.  The
  transformation $\widetilde{F}$ acts by moving points in
  $A_n$, $B_n$ into the scaled images of $T(A)$ and $T(B)$ in
  $X \times \{n+1\}$, and points in $C_n$ and $D_n$ are moved
  into the scaled images of $T(C)$ and $T(D)$ in $X \times
  \{n-1\}$.  See Figure~\ref{fig:discrete_maharam}.
\end{example}
  
  \begin{figure}[h!]
    \centering
        \begin{tikzpicture}[xscale=4]
      \begin{scope}
        \node[scale=0.8] at (-0.4, 0) {$X \times \{n\}$};


        \foreach \x in {0, 0.242, 0.333, 0.707, 1} {
          \draw (\x, -0.2) -- (\x, 0.2);
        }
        

        \draw[ultra thick, black!25!white] (0, 0) -- node[above, scale=0.5] {$A_n$} (0.242, 0); 
        \draw[ultra thick, black!75!white] (0.242, 0) -- node[above, scale=0.5]
        {$B_n$} (0.333, 0); 
        \draw[ultra thick, black!50!white] (0.333, 0) -- node[above,
          scale=0.5] {$C_n$} (0.707, 0); 
        \draw[ultra thick, black] (0.707, 0) -- node[above,
          scale=0.5] {$D_n$} (1, 0); 

        \draw[->] (.121, 0.5) .. controls (0.121, 1) and (0.379, 1.5)
        .. (0.379, 2);
        \draw[->] (0.2875, 0.5) .. controls (0.2875, 1) and (0.2123,
        1.5) .. (0.2123, 2);

        \draw[->] (0.52, -0.5) .. controls (0.52, -1) and (0.479,
        -1.5) .. (0.479, -2);
        \draw[->] (0.8535, -0.5) .. controls (0.8535, -1) and (0.073,
        -1.5) .. (0.073, -2);

      \end{scope}

      \begin{scope}[yshift=1in,xscale=0.5]
        \node[scale=0.8] at (-1, 0) {$X \times \{n+1\}$};
        \foreach \x in {0, 0.333, 0.516, 1} {
          \draw (\x, -0.2) -- (\x, 0.2);
        }

        \draw[gray] (0, 0) -- (0.33, 0);
        \draw[ultra thick, black!25!white] (0.516, 0) -- node[above=4pt, scale=0.5]
             {$\widetilde{F}(A_n)$} (1, 0); 
        \draw[ultra thick, black!75!white] (0.333, 0) -- node[above=4pt, scale=0.5]
             {$\widetilde{F}(B_n)$} (0.516, 0);
      \end{scope}

      \begin{scope}[yshift=-1in,xscale=2]
        \node[scale=0.8] at (-0.25, 0) {$X \times \{n-1\}$};
        \foreach \x in {0, 0.146, 0.333, 1} {
          \draw (\x, -0.2) -- (\x, 0.2);
        }

        \draw[gray] (0.33, 0) -- (1, 0);
        \draw[ultra thick, black!50!white] (0.146, 0) -- node[above, scale=0.5]
             {$\widetilde{F}(C_n)$} (0.333, 0); 
        \draw[ultra thick, black] (0, 0) -- node[above, scale=0.5]
             {$\widetilde{F}(D_n)$} (0.146, 0); 
      \end{scope}

      \begin{scope}[yshift=-1.5in]
        \node at (0, 0) {$\vdots$};
      \end{scope}

      \begin{scope}[yshift=1.5in]
        \node at (0, 0) {$\vdots$};
      \end{scope}

    \end{tikzpicture}
    \caption{The discrete Maharam extension of the reciprocal
      transformation described in Example~\ref{ex:pair_rotation}.}
    \label{fig:discrete_maharam}
  \end{figure}

That is, we apply $F$ to each copy of $X$, but then move points up or
down a level depending on how $F$ scales at that point (i.e., based on
the value of $\frac{dF_*\mu}{d\mu}$).  Those points in $T^{-1}(S)$ will first be moved by
$T$ into $S$, then scaled by $\rho$ when $\Phi$ is applied.  By moving
these points up or down one level (to a copy of $X$ whose
$\widetilde{\mu}$-measure shrinks by $\rho^{-1}$ or grows by $\rho$), we effectively
cancel out the scaling effect of $\Phi$.  That is, the measure
$\widetilde{\mu}$ we have defined is preserved by $\widetilde{F}$.
Before proving this we show that $\widetilde{F}$ is always bijective
so that we may consider images of sets instead of their preimages when
proving $\widetilde{\mu}$ is preserved.

\begin{lemma}
  \label{lemma:FT_bijective}
  The discrete Maharam extension $\widetilde{F}$ associated to a
  reciprocal transformation $F = \Phi T$ is bijective.
\end{lemma}
\begin{proof}
  First observe that for each integer $m$, the preimage
  $\widetilde{F}^{-1}(S \times \{m\})$ is $F^{-1}(S) \times \{m +
  1\}$ and $F^{-1}(S) = (\Phi T)^{-1}(S) = T^{-1} \Phi(S)$.
  Similarly, $F^{-1}(\Phi (S) \times \{m\}) = F^{-1}(\Phi(S)) \times
  \{m - 1\}$ and $F^{-1}(\Phi(S)) = (\Phi T)^{-1} \Phi(S) = T^{-1}\Phi
  \Phi(S) = T^{-1}(S)$.  As $F$ is a bijection, $\widetilde{F}$
  restricted to $T^{-1} \Phi(S) \times \{m + 1\}$ is a bijection onto
  $S \times \{m\}$, and similarly $\widetilde{F}$ restricted to
  $T^{-1}(S) \times \{m - 1\}$ is a bijection onto $\Phi(S) \times
  \{m\}$.  As each level of $X \times \mathbb{Z}$ is decomposted into
  a disjoint union of $S \times \{m\}$ and $\Phi(S) \times \{m\}$
  sets on which $\widetilde{F}$ is bijective, this shows that
  $\widetilde{F}$ is a bijection.
\end{proof}

\begin{proposition}
  \label{prop:FT_preservesMuT}
  The map $\widetilde{F}$ preserves the measure $\widetilde{\mu}$.
\end{proposition}
\begin{proof}
  For any $E \subseteq X \times \mathbb{Z}$ and any $n \in
  \mathbb{Z}$, we break the $n$-th level $E_n$ up into two pieces:
  the portion which $\widetilde{F}$ will move up a level (denoted $E_n^+$), and the portion
  which will move down a level ($E_n^-$).  That is,
  \begin{align*}
    E_n^+ &:= \left(\pi(E_n) \cap T^{-1}(S)\right) \times \{n\}, \text{
      and } \\
    E_n^- &:= \left(\pi(E_n) \cap T^{-1}(\Phi(S))\right) \times \{n\}.
  \end{align*}
  The $n$-th level of the image $\widetilde{F}(E)$ is
  $\widetilde{F}(E_{n-1}^+) \sqcup \widetilde{F}(E_{n+1}^-)$, and so
  has measure
  \begin{align*}
    & \rho^{-n}
    \left(\mu\left(\pi\left(\widetilde{F}\left(E_{n-1}^+\right)\right)\right)
    +
    \mu\left(\pi\left(\widetilde{F}\left(E_{n+1}^-\right)\right)\right)\right)
    \\
    =& \rho^{-n}
    \left(\mu\left(F\left(\pi\left(E_{n-1}^+\right)\right)\right) + \mu\left(F\left(\pi\left(E_{n+1}^-\right)\right)\right)\right)
  \end{align*}
  However, by definition of the sets $E_{n \pm 1}^{\mp}$,
  \begin{align*}
    \mu\left(\pi\left(F\left(E_{n-1}^+\right)\right)\right) &= \rho
    \mu\left(\pi\left(E_{n-1}^+\right)\right), \text{ and } \\
    \mu\left(\pi\left(F\left(E_{n+1}^-\right)\right)\right) &= \rho^{-1}\mu\left(\pi\left(E_{n+1}^-\right)\right).
  \end{align*}
  Thus the $n$-th level of $\widetilde{F}(E)$ has measure
  \[
  \rho^{-(n-1)}\mu\left(\pi\left(E_{n-1}^+\right)\right) +
  \rho^{-(n+1)}\mu\left(\pi\left(E_{n+1}^-\right)\right).
  \]
  That is, the terms of
  \[
  \widetilde{\mu}\left(\widetilde{F}\left(E\right)\right)
  =
  \sum_{n\in\mathbb{Z}} \rho^{-n}\mu\left(\pi\left(\widetilde{F}\left(E_n\right)\right)\right)
  \]
  are the same as the terms of
  \[
  \widetilde{\mu}(E) = \sum_{n\in\mathbb{Z}}\rho^{-n}\mu\left(\pi\left(E_n\right)\right),
  \]
  just slightly reordered.  (Note that the terms of all our series are
  non-negative, and so these series either diverge to infinity or are
  absolutely convergent.  In either case, reordering the terms of the
  series does not change the series' value.)
\end{proof}

A natural question to consider is whether ergodicity of a reciprocal
transformation $F$ is sufficient to ensure the ergodicity of its discrete
Maharam extension $\widetilde{F}$.  Using
Lemma~\ref{lemma:Fn_not_ergodic} and
Proposition~\ref{prop:Ftilde_not_ergodic} below,
Example~\ref{ex:Ftilde_not_ergodic} gives a counterexample to the
reasonable-seeming conjecture that ergodicity of $F$ implies
ergodicity of $\widetilde{F}$.

\begin{lemma}
  \label{lemma:Fn_not_ergodic}
  Suppose $F = \Phi T$ is a reciprocal transformation where some
  iterate $F^n$ is measure-preserving.  Then the associated discrete
  Maharam extension can not be ergodic.
\end{lemma}
\begin{proof}
  As $F^n$ is measure-preserving and each iterate of $F$ modifies the
  measure of a subset by a factor of $\rho$ or $\rho^{-1}$, for each
  $0 \leq m \leq n$
  the Radon-Nikodym derivative $\frac{dF^m_*\mu}{d\mu}$ is bounded below
  by $\rho^{-n}$ and above by $\rho^n$.  Thus any point of $X \times
  \{0\}$ can not visit any level $X \times \{k\}$ for $|k| > n$, as it
  must return to $X \times \{0\}$ at the $n$-th iterate.
\end{proof}

\begin{proposition}
  \label{prop:Ftilde_not_ergodic}
  If $F = \Phi T$ is a reciprocal transformation such that $\mu(S_1)$
  (in the notation of Lemma~\ref{lemma:conservativeFS}) has measure
  zero, then the discrete Maharam extension $\widetilde{F}$ is not
  ergodic. 
\end{proposition}
\begin{proof}
  Recall $\mu(F_S(E)) = \rho^{n-2} \mu(E)$ for each $E \subseteq S_n$.
  If $\mu(S_1) = 0$, then subsets of $S$ can never decrease in
  measure, as entering into $S_1$ is the only way that the image of a
  set under $F_S$ can decrease its measure.  By
  Lemma~\ref{lemma:conservativeFS}, $F_S$ is conservative, and so if
  some $S_n$ for $n > 2$ had positive measure, almost every point in
  $S_n$ must return to $S_n$ infinitely-often.  Thus if some $S_n$, $n
  > 2$, had positive measure, any subset of $S_n$ of positive measure
  must have its measure increase by a factor of $\rho^{n-2}$
  infinitely often.  However, $S$ has finite measure, and so if
  $\mu(S_1) = 0$ then $S = S_2$ up to a null set.  This implies $F^2$
  is in fact a measure-preserving transformation, and by
  Lemma~\ref{lemma:Fn_not_ergodic}, the associated discrete Maharam
  extension can not be ergodic.
\end{proof}

\begin{example}
  \label{ex:Ftilde_not_ergodic}
  Consider our running example of a reciprocal transformation from
  Example~\ref{ex:pair_rotation} and Example~\ref{ex:ergodic_aiet},
  and its associated discrete Maharam extension in
  Example~\ref{ex:discrete_maharam_example}.  As previously noted,
  when $\mu(B) + \rho^{-1}\mu(D)$ is irrational, the reciprocal transformation $F$ is
  ergodic.  Notice, in the notation of
  Lemma~\ref{lemma:conservativeFS}, that $S_1 = \emptyset$ and $S_2 =
  S$.  Hence the associated discrete Maharam extension from
  Example~\ref{ex:discrete_maharam_example} is not ergodic.
\end{example}

We thus see that ergodicity of $\widetilde{F}$ requires a more
restrictive condition than just ergodicity of $F$.  In order to relate
the ergodicity of $\widetilde{F}$ to that of the reciprocal
transformation $F = \Phi T$, we must first describe the ``Krieger
type'' and ratio set of non-singular transformations.

The study of Krieger types, the closely related notion of the ratio
set, and non-singular ergodic theory in general has its origin in a
simple question about whether a given measurable transformation
necessarily has any invariant measures.  The work of Murray and von
Neumann had established a connection between ergodic
theory and operator algebras \cite{MurrayVonNeumann}, which was utilized by Dye in
his study of orbit equivalence \cite{Dye}.  Shortly afterwards,
Ornstein provided the first example of a measurable transformation not
preserving any equivalent measure \cite{Ornstein}.  This initiated the study of the
various types of non-singular transformations that could exist, and
would lead to Krieger's classification of non-singular transformations
motivated by the existing classification of von Neumann algebras
\cite{Krieger12}.

Associated to each non-singular transformation $F$ is a collection of
non-negative real numbers called the transformation's \emph{ratio
set}, denoted $r(F)$.  The number $\lambda$ is an element of $r(F)$ if
for every set $E$ of positive measure and every $\epsilon > 0$ there
exists a subset $E' \subseteq E$ of positive measure and an integer
$n$ so that for every $x \in E'$, $F^n(x) \in E'$ and $\left| \frac{d
  F^n_*\mu}{d\mu}(x) - \lambda \right| < \epsilon$.  That is, the
ratio set essentially measures all of the proportions by which all
sets get scaled.  Computing the ratio set of a transformation is an
extremely difficult problem in general, though can sometimes be
accomplished in certain special settings; e.g., see
\cite{HamachiOsikawa},
\cite{ChoskiHawkinsPrasad}
\cite{Danilenko},
\cite{DooleyHawkinsRalston}, and
\cite{Furno}.

Non-singular transformations are characterized into various
\emph{Krieger types} based on their ratio set.  We say an ergodic
transformation $F$ has \emph{(Krieger) type $II$} if its ratio set is
$r(F) = \{1\}$, and \emph{(Krieger) type $III_\lambda$} if its ratio
set is $r(F) = \{0\} \cup \{\lambda^n \, \big| \, n \in \mathbb{Z}\}$
for a real number $\lambda > 0$.  It was shown in \cite{Krieger} that
type $II$ transformations always admit an equivalent invariant
measure, and correspondingly there are two subtypes: type $II_1$
transformations admit an invariant probability measure, whereas type
$II_\infty$ transformations admit an infinite invariant measure.  Type
$III$ transformations do not admit any invariant measures, and
similarly have a breakdown into subtypes.

The subtypes of type $III$ transformations are
described as being type $III_\lambda$ where $\lambda \in [0, 1]$.  Of
these subtypes there are three particular cases corresponding to $r(F)
= \{0, 1\}$ (type $III_0$), $r(F) = [0, \infty)$ (type $III_1$), and
  $r(F) = \{ \lambda^n \, \big| \, n \in \mathbb{Z}\}$ with $\lambda
  \in (0, 1)$.
Theorem~\ref{thm:Ftilde_ergodic} below shows that $\widetilde{F}$ is
ergodic (with respect to $\widetilde{\mu}$) 
precisely when $F$ is ergodic (with respect to $\mu$) is type
$III_\lambda$ where $\lambda = \frac{1}{\rho}$.


\begin{theorem}
  \label{thm:Ftilde_ergodic}
  Let $T : X \to X$ be a measure-preserving transformation of a
  probability space $(X, \mathscr{B}, \mu)$, $\Phi : X \to X$ a
  scaling involution with scaling ratio $\rho > 1$, $F = \Phi T$ their
  composition, and $\widetilde{F} : X \times \mathbb{Z} \to X \times
  \mathbb{Z}$ the corresponding discrete Maharam extension with
  measure $\widetilde{\mu}$.  The map
  $\widetilde{F}$ is $\widetilde{\mu}$-ergodic if and only if 
  $F$ is ergodic and type $III_{1/\rho}$.
\end{theorem}
\begin{proof}
  Suppose that $F$ is ergodic and type $III_{1/\rho}$.  To establish
  that $\widetilde{F}$ is ergodic, we will show that for every subset of
  $X \times \mathbb{Z}$ of positive measure,
  $\widetilde{\mu}$-almost every point $(x_0, n_0) \in X \times
  \mathbb{Z}$ eventually enters the set.

  Let $E \subseteq X \times \mathbb{Z}$ be any set of positive
  $\widetilde{\mu}$ measure.  As $E = \bigsqcup_{k \in \mathbb{Z}} E_k$ with $E_k
  = E \cap (X \times \{k\})$, there must exist some $k$ with
  $\widetilde{\mu}(E_k) > 0$.
  Let $U \subseteq X \times \mathbb{Z}$ consist of those points $(x,
  n)$ so that $\widetilde{F}^N(x, n) \notin E_k$ for all $N$.  Observe
  that if $U$ has positive $\widetilde{\mu}$-measure, then $\pi(U)$ must have positive
  $\mu$-measure.
  By the assumption that $F$ is ergodic, for $\mu$-almost every point
  $x \in \pi(U)$ there would exist some integer $m$ so that $F^m(x)
  \in \pi(E_k)$.

  For each integer $j \geq 0$, let
  $V_j \subseteq \pi(U)$ consist of those points $x \in \pi(U)$ so
  that $F^j(x) \in \pi(E_k)$ but $F^\ell(x) \notin \pi(E_k)$ for each
  $\ell < j$.  We may further decompose each $V_j$ into subsets $V_j^q$
  where $\frac{dF^j_*\mu}{d\mu}$ equals $\rho^q$ for $\mu$-a.e. $x \in
  V_j^q$.  Observe if $x \in V_j^q$, then $\widetilde{F}^j(x, k - q)
  \in E_k$.  As $\pi(U) = \bigcup_{j,q} V_j^q$ is a countable union,
  at least one $V_j^q$ has positive measure.  This would mean $U \cap
  \pi^{-1}(V_j^q)$ has positive measure, and so a positive measure
  subset of $U$ enters $E_k$, contradicting the definition of $U$.
  Thus $U$ must have had zero measure, and so almost every point of $X
  \times \mathbb{Z}$ eventually enters every set of positive measure,
  and $\widetilde{F}$ is ergodic.

  Suppose now that $\widetilde{F}$ is ergodic.  We will first show
  that $F$ must be ergodic as well by showing an invariant subset $E
  \subseteq X$ has full measure or zero measure.  Notice that
  $\widetilde{E} = \pi^{-1}(E)$ is $\widetilde{F}$-invariant.  Hence
  $\widetilde{E}$ or its complement has zero measure.  Thus either $E$
  or its complement has zero measure as well, since $\pi$ is
  non-singular.

  Now we  must establish that $F$ is type $III_{1/\rho}$.  Let $E$ be any
  subset of $X$ of positive measure, and consider $E \times \{0\}$.
  By ergodicity of $\widetilde{F}$, for each integer $q$ and each $x
  \in E \times \{0\}$ there exists an $m$ so that $\widetilde{F}^m(x, 0) \in E
  \times \{q\}$.  As $F$ is semiconjugate to
  $\widetilde{F}$, this means $F^m(x) \in E$ and
  $\frac{dF_*\mu}{d\mu}(x) = \rho^q$.  Thus all powers of $\rho$ are
  in the ratio set of $F$, and $F$ is type $III_{1/\rho}$.
\end{proof}

\section{Directions for Future Work}
\label{sec:Future}
As previously noted, determining the ratio set of a given
transformation seems to be a very difficult problem in general.  In
order for Theorem~\ref{thm:Ftilde_ergodic} to be useful, we would like
to have some easy-to-check condition that will guarantee a reciprocal
transformation has Krieger type $III_{1/\rho}$.  One possible avenue which
the author is currently pursuing has to do with the case that the
reciprocal transformation $F$ is ``affinely self-similar'' to its
first-return map $F_S$.  To be precise, suppose that $X$ is an
interval, $T$ is an interval exchange, and $\Phi$ is a piecewise
affine involution of $X$.  The reciprocal transformation $F = \Phi T$
is then an affine interval exchange, and it is not difficult to see
that the first-return map $F_S$ is also an affine interval exchange
transformation.  In the special case that $F_S : S \to S$ is simply a
rescaled copy of $F : X \to X$ (this is what we mean by ``affinely
self-similar''), it may be possible to show that $F$ has Krieger type
$III_{1/\rho}$.  In particular, the self-similarity allows us to obtain
scaled copies of the smaller set $S$ into which almost every point of
$X$ must enter.  It may be that this existence of smaller and smaller
subsets, each scaled by a factor of $\rho^{-1}$, will imply that all
sets have subsets that must eventually scale by $\rho^{-n}$ as the
points visit these scaled copies of $S$.

At the moment the author does not know of an example of an interval
exchange $T$ and scaling involution $\Phi$ whose associated reciprocal
transformation $F = \Phi T$ would exhibit this self-similarity.  It is
possible that in fact no such transformation exists, and so simply
establishing the existence or non-existence of these affinely
self-similar transformations seems interesting.  To this end, it would
be helpful to have a better understanding of the moduli space of
reciprocal transformations, at least in the case when $X$ is an
interval and $T$ is an interval exchange with some fixed number of
intervals of continuity.  Once this moduli space is understood, a
transformation on that space similar to Rauzy-Veech induction for
traditional interval exchanges might be helpful to search for
self-similar reciprocal transformations.  To date no such procedure
exists for general affine interval exchanges, but perhaps in this
restricted case such a procedure could be defined.


\begin{thebibliography}{DHR11}

\bibitem[Aar97]{Aaronson}
Jon Aaronson, \emph{An introduction to infinite ergodic theory}, Mathematical
  Surveys and Monographs, vol.~50, American Mathematical Society, Providence,
  RI, 1997. \MR{1450400}

\bibitem[CHP87]{ChoskiHawkinsPrasad}
J.~R. Choksi, J.~M. Hawkins, and V.~S. Prasad, \emph{Abelian cocycles for
  nonsingular ergodic transformations and the genericity of type {${\rm
  III}_1$} transformations}, Monatsh. Math. \textbf{103} (1987), no.~3,
  187--205. \MR{894170}

\bibitem[Dan04]{Danilenko}
Alexandre~I. Danilenko, \emph{Infinite rank one actions and nonsingular
  {C}hacon transformations}, Illinois J. Math. \textbf{48} (2004), no.~3,
  769--786. \MR{2114251}

\bibitem[DHR11]{DooleyHawkinsRalston}
A.~H. Dooley, J.~Hawkins, and D.~Ralston, \emph{Families of type {${\rm
  III}_0$} ergodic transformations in distinct orbit equivalent classes},
  Monatsh. Math. \textbf{164} (2011), no.~4, 369--381. \MR{2861592}

\bibitem[DS23]{DanilenkoSilva}
Alexandre~I. Danilenko and Cesar~E. Silva, \emph{Ergodic theory: nonsingular
  transformations}, Ergodic theory, Encycl. Complex. Syst. Sci., Springer, New
  York, [2023] \copyright 2023, pp.~233--292. \MR{4647080}

\bibitem[Dye59]{Dye}
H.~A. Dye, \emph{On groups of measure preserving transformations. {I}}, Amer.
  J. Math. \textbf{81} (1959), 119--159. \MR{131516}

\bibitem[Fur14]{Furno}
Joanna Furno, \emph{Orbit equivalence of {$p$}-adic transformations and their
  iterates}, Monatsh. Math. \textbf{175} (2014), no.~2, 249--276. \MR{3260869}

\bibitem[HO81]{HamachiOsikawa}
Toshihiro Hamachi and Motosige Osikawa, \emph{Ergodic groups of automorphisms
  and {K}rieger's theorems}, Seminar on Mathematical Sciences, vol.~3, Keio
  University, Department of Mathematics, Yokohama, 1981. \MR{617740}

\bibitem[Kri69]{Krieger12}
Wolfgang Krieger, \emph{On non-singular transformations of a measure space.
  {I}, {II}}, Z. Wahrscheinlichkeitstheorie und Verw. Gebiete \textbf{11}
  (1969), 83--97; ibid. 11 (1969), 98--119. \MR{240279}

\bibitem[Kri70]{Krieger}
\bysame, \emph{On the {A}raki-{W}oods asymptotic ratio set and non-singular
  transformations of a measure space}, Contributions to {E}rgodic {T}heory and
  {P}robability ({P}roc. {C}onf., {O}hio {S}tate {U}niv., {C}olumbus, {O}hio,
  1970), Lecture Notes in Math., Vol. 160, Springer, Berlin-New York, 1970,
  pp.~158--177. \MR{414823}

\bibitem[Mah64]{Maharam}
D.~Maharam, \emph{Incompressible transformations}, Fund. Math. \textbf{56}
  (1964), 35--50. \MR{169988}

\bibitem[MvN43]{MurrayVonNeumann}
F.~J. Murray and J.~von Neumann, \emph{On rings of operators. {IV}}, Ann. of
  Math. (2) \textbf{44} (1943), 716--808. \MR{9096}

\bibitem[Orn60]{Ornstein}
Donald~S. Ornstein, \emph{On invariant measures}, Bull. Amer. Math. Soc.
  \textbf{66} (1960), 297--300. \MR{146350}

\bibitem[Pet89]{Petersen}
Karl Petersen, \emph{Ergodic theory}, Cambridge Studies in Advanced
  Mathematics, vol.~2, Cambridge University Press, Cambridge, 1989, Corrected
  reprint of the 1983 original. \MR{1073173}

\bibitem[Via06]{Viana}
Marcelo Viana, \emph{Ergodic theory of interval exchange maps}, Rev. Mat.
  Complut. \textbf{19} (2006), no.~1, 7--100. \MR{2219821}

\bibitem[Vor96]{Vorobets}
Ya~B Vorobets, \emph{Planar structures and billiards in rational polygons: the
  veech alternative}, Russian Mathematical Surveys \textbf{51} (1996), no.~5,
  779.

\bibitem[Wal82]{Walters}
Peter Walters, \emph{An introduction to ergodic theory}, Graduate Texts in
  Mathematics, vol.~79, Springer-Verlag, New York-Berlin, 1982. \MR{648108}

\bibitem[Wri15]{Wright}
Alex Wright, \emph{Translation surfaces and their orbit closures: an
  introduction for a broad audience}, EMS Surv. Math. Sci. \textbf{2} (2015),
  no.~1, 63--108. \MR{3354955}

\bibitem[Zor06]{Zorich}
Anton Zorich, \emph{Flat surfaces}, arXiv preprint math/0609392 (2006).

\end{thebibliography}
\providecommand{\bysame}{\leavevmode\hbox to3em{\hrulefill}\thinspace}
\providecommand{\MR}{\relax\ifhmode\unskip\space\fi MR }
\providecommand{\MRhref}[2]{%
  \href{http://www.ams.org/mathscinet-getitem?mr=#1}{#2}
}
\providecommand{\href}[2]{#2}

\end{document}